\documentclass[11pt]{article}

\usepackage[T1]{fontenc}
\usepackage{lmodern}
\usepackage{microtype}
\usepackage{amsmath,amssymb,amsthm,mathtools}
\usepackage{booktabs}
\usepackage{enumitem}
\usepackage[hidelinks]{hyperref}
\hypersetup{
  pdftitle={Median--Extremes Alternation: A Parity-Unified Maximal-Contrast Traversal of Finite Ordered Sets},
  pdfauthor={David Carr},
  pdfsubject={A uniqueness theorem for Median--Extremes Alternation on finite ordered sets},
  pdfkeywords={finite ordered sets, permutations, greedy traversal, bilateral symmetry, maximal contrast, orientation}
}

\newtheorem{theorem}{Theorem}[section]
\newtheorem{proposition}[theorem]{Proposition}
\newtheorem{corollary}[theorem]{Corollary}
\theoremstyle{definition}
\newtheorem{definition}[theorem]{Definition}
\newtheorem{example}[theorem]{Example}
\theoremstyle{remark}

\newcommand{\MEA}{\operatorname{MEA}}
\newcommand{\ord}{\operatorname{ord}}

\title{Median--Extremes Alternation:\\
A Parity-Unified Maximal-Contrast Traversal of Finite Ordered Sets}
\author{David Carr}
\date{July 2026}

\begin{document}

\maketitle

\begin{abstract}
Median--Extremes Alternation (MEA) is a deterministic traversal of a finite linearly ordered set. The positions are folded about their midpoint into bilateral shells, and each shell is assigned an integer-valued normalized radius. Beginning with the central shell, the traversal repeatedly selects the unvisited shell having greatest radial contrast with the shell most recently engaged. We prove that this local rule has a strict unique maximizer at every step and forces the shell order
\[
0,\ q,\ 1,\ q-1,\ 2,\ q-2,\ldots,
\]
for both odd and even cardinalities. Parity affects only whether the central shell is a singleton or a pair. A separately supplied global orientation determines the order of the elements within every shell, yielding exactly two mirror traversals for $n\geq 2$ before orientation is fixed and exactly one afterward. Explicit formulas, an $O(n)$ generation algorithm, examples, and scope conditions are given. The result replaces an earlier fixed-center distance-minimization formulation, which does not generate the canonical MEA traversal.
\end{abstract}


\section{Introduction}

Let $[n]=\{1,2,\ldots,n\}$ carry its usual linear order. Median--Extremes Alternation (MEA) traverses $[n]$ by first engaging its median position or median pair, then its exterior pair, and then alternating between the inner and outer boundaries of the unresolved positions. For example,
\[
\MEA_{7,+}=(4,1,7,3,5,2,6)
\]
and
\[
\MEA_{10,+}=(5,6,1,10,4,7,2,9,3,8).
\]
The subscript $+$ records the orientation used inside every two-element block; orientation is treated separately in Section~\ref{sec:orientation}.

The purpose of this paper is to characterize the shell order by an independent local invariant rather than merely restating the traversal. The correct invariant is \emph{maximal radial contrast from the shell most recently engaged}. It is not minimum distance from the fixed center, and it is not maximum distance from the most recently visited raw index. The optimization is performed on bilateral shells treated as aggregate objects.

A preliminary formulation characterized MEA by nondecreasing distance from the fixed center and asserted distinct odd and even branches. That formulation is incorrect. Canonical MEA jumps from the center to the exterior shell, so its fixed-center distances are not nondecreasing. The definitions and theorem below replace that formulation with one rule applying uniformly across parity.

The main result is elementary but exact: once the center is fixed, greedy maximal radial contrast forces a unique shell sequence. A global orientation parameter then converts the shell sequence into a unique traversal of the underlying ordered set.

Distance-based traversal problems appear in several established forms. Gonzalez's farthest-first method selects a point by its distance from the set already chosen in metric clustering~\cite{gonzalez1985}. Maximum-scatter Hamiltonian paths and tours instead optimize separation between consecutive points~\cite{arkin1999}; line-based constructions from that literature also underlie later grid algorithms~\cite{hoffmann2017}. MEA is neither objective: it first groups the ordered domain into bilateral shells, fixes the central shell as the initial condition, and greedily maximizes radius difference only from the shell most recently engaged. These references locate the construction among nearby optimization ideas rather than supplying premises for the theorem.

\section{Bilateral shells and normalized radius}

Fix $n\geq 1$ and write
\[
a=\left\lfloor\frac{n+1}{2}\right\rfloor,
\qquad
b=\left\lceil\frac{n+1}{2}\right\rceil,
\qquad
q=\left\lfloor\frac{n-1}{2}\right\rfloor.
\]
Thus $a=b$ when $n$ is odd and $b=a+1$ when $n$ is even.

\begin{definition}[Bilateral shells]
For each $r\in\{0,1,\ldots,q\}$, define
\[
B_r=\{a-r,\ b+r\}.
\]
When $n$ is odd, $B_0=\{a\}$ is a singleton; otherwise every $B_r$, including $B_0$, has two elements. The family
\[
\mathcal{B}_n=\{B_0,B_1,\ldots,B_q\}
\]
is called the \emph{bilateral shell partition} of $[n]$.
\end{definition}

The shells are the orbits of reflection about the midpoint, grouped by distance from the center. They partition $[n]$.

\begin{definition}[Normalized shell radius]
The normalized radius of $B_r$ is
\[
\rho(B_r)=r.
\]
Equivalently, if $d_r$ denotes the ordinary geometric distance of the elements of $B_r$ from the midpoint $(n+1)/2$, then
\[
d_r=r+\delta_n,
\qquad
\delta_n=
\begin{cases}
0, & n\text{ odd},\\[2pt]
\frac12, & n\text{ even},
\end{cases}
\]
and $\rho(B_r)=d_r-\delta_n$.
\end{definition}

The normalization simply translates all even-case radii by $-1/2$. Consequently,
\[
|\rho(B_r)-\rho(B_s)|=|d_r-d_s|,
\]
so no radial contrast is altered. The sole purpose of normalization is to assign radius $0$ to the central block in both parities.

\begin{definition}[Shell traversal]
A shell traversal is a permutation
\[
(r_0,r_1,\ldots,r_q)
\]
of $\{0,1,\ldots,q\}$. It represents the block sequence
\[
(B_{r_0},B_{r_1},\ldots,B_{r_q}).
\]
\end{definition}

\section{The maximal-radial-contrast theorem}

\begin{definition}[Maximal radial contrast]
A shell traversal $(r_0,r_1,\ldots,r_q)$ satisfies the \emph{maximal-radial-contrast rule} when
\begin{enumerate}[label=\textup{(\roman*)}]
    \item $r_0=0$;
    \item for each $k\geq 1$, the next radius maximizes contrast with the radius most recently visited:
    \[
    r_k\in\underset{r\notin\{r_0,\ldots,r_{k-1}\}}{\arg\max}
    |r-r_{k-1}|.
    \]
\end{enumerate}
\end{definition}

The rule refers to the previous \emph{shell radius}, not to the previous raw position in $[n]$.

\begin{theorem}[Maximal-contrast uniqueness]\label{thm:main}
For every $n\geq 1$, there is exactly one shell traversal satisfying the maximal-radial-contrast rule. It is
\[
0,\ q,\ 1,\ q-1,\ 2,\ q-2,\ldots.
\]
More explicitly, $r_0=0$ and, for $1\leq k\leq q$,
\[
r_k=
\begin{cases}
q-\dfrac{k-1}{2}, & k\text{ odd},\\[8pt]
\dfrac{k}{2}, & k\text{ even}.
\end{cases}
\]
At every nonterminal step the maximizing shell is unique.
\end{theorem}

\begin{proof}
The case $q=0$ is immediate. Assume $q\geq 1$.

The traversal starts at radius $0$. Among the unvisited radii $\{1,\ldots,q\}$, the unique maximizer of $|r-0|$ is $q$. Thus the first two radii are $0,q$.

After radius $q$ is visited, the unvisited radii form the contiguous interval $[1,q-1]$, and the current radius $q$ lies immediately above it. The unique radius in that interval farthest from $q$ is its lower endpoint $1$. After visiting $1$, the remaining radii form $[2,q-1]$, and the current radius lies immediately below that interval. The unique farthest radius is now its upper endpoint $q-1$.

The same structure persists inductively. At any nonterminal stage, suppose the unvisited radii form a nonempty contiguous interval $[L,U]$ and the current radius is adjacent to it.

If the current radius is $U+1$, then for every $r\in[L,U]$,
\[
|(U+1)-L|\geq |(U+1)-r|,
\]
with equality only when $r=L$. Hence $L$ is the unique maximizer. Removing $L$ leaves the interval $[L+1,U]$, and the new current radius $L$ lies immediately below it.

If the current radius is $L-1$, then for every $r\in[L,U]$,
\[
|U-(L-1)|\geq |r-(L-1)|,
\]
with equality only when $r=U$. Hence $U$ is the unique maximizer. Removing $U$ leaves the interval $[L,U-1]$, and the new current radius $U$ lies immediately above it.

Thus contiguity of the remainder and boundary adjacency of the current radius are preserved. The traversal is therefore forced to remove alternately the upper and lower endpoints of the remaining interval:
\[
q,1,q-1,2,q-2,3,\ldots.
\]
Including the initial center radius gives the stated sequence. Since the maximizer is strict at every nonterminal step, no alternative shell traversal satisfies the rule.
\end{proof}

\begin{corollary}[Descending contrast profile]\label{cor:contrast}
For the traversal in Theorem~\ref{thm:main},
\[
|r_k-r_{k-1}|=q-k+1,
\qquad 1\leq k\leq q.
\]
Hence the successive radial contrasts are
\[
q,q-1,q-2,\ldots,1.
\]
\end{corollary}

\begin{corollary}[Alternating endpoint elimination]
After the center is engaged, MEA is equivalently generated by repeatedly removing the upper endpoint and then the lower endpoint of the remaining interval of shell radii:
\[
[1,q]\longmapsto q,1,q-1,2,q-2,3,\ldots.
\]
This endpoint alternation is a consequence of maximal radial contrast, not an additional selection rule.
\end{corollary}

\section{Orientation and the index traversal}\label{sec:orientation}

Theorem~\ref{thm:main} determines which shell is visited next. It does not determine which element of a two-element shell is visited first. That second choice is governed by a single global orientation parameter supplied by the application domain.

\begin{definition}[Global orientation]
Let $\omega\in\{+,-\}$. For a shell $B_r=\{a-r,b+r\}$, define its oriented expansion by
\[
\ord_{+}(B_r)=(a-r,b+r),
\qquad
\ord_{-}(B_r)=(b+r,a-r),
\]
with duplicate entries removed when $B_0$ is a singleton.
\end{definition}

Thus $+$ means low index before high index in every block, while $-$ means high index before low index in every block. The same $\omega$ applies uniformly to the center pair and every bilateral shell. Orientation is relational to the operative order of the domain; the theorem does not declare one direction intrinsically prior to the other.

\begin{definition}[Oriented MEA traversal]
Let $(r_0,\ldots,r_q)$ be the unique shell sequence from Theorem~\ref{thm:main}. Define
\[
\MEA_{n,\omega}
=
\ord_{\omega}(B_{r_0})
\mathbin{\Vert}
\ord_{\omega}(B_{r_1})
\mathbin{\Vert}\cdots\mathbin{\Vert}
\ord_{\omega}(B_{r_q}),
\]
where $\Vert$ denotes concatenation.
\end{definition}

\begin{corollary}[Oriented uniqueness]\label{cor:oriented}
For fixed $n$ and fixed $\omega$, $\MEA_{n,\omega}$ is the unique traversal of $[n]$ satisfying all of the following:
\begin{enumerate}[label=\textup{(\roman*)}]
    \item the central block is visited first;
    \item every bilateral shell is visited as one contiguous aggregate block;
    \item shell selection obeys maximal normalized radial contrast from the shell most recently engaged;
    \item every two-element block is traversed consistently according to the same global orientation $\omega$.
\end{enumerate}
\end{corollary}

\begin{proof}
Theorem~\ref{thm:main} uniquely fixes the shell sequence. Fixed $\omega$ uniquely fixes the internal order of every block, including the central block when $n$ is even. Concatenation therefore yields one traversal.
\end{proof}

Without a supplied orientation, there are exactly two globally coherent traversals for $n\geq 2$, namely $\MEA_{n,+}$ and $\MEA_{n,-}$. They are mirror images.

\begin{proposition}[Mirror conjugacy]
Let $\sigma_n(i)=n+1-i$. Applying $\sigma_n$ elementwise to $\MEA_{n,+}$ gives $\MEA_{n,-}$:
\[
\MEA_{n,-}=\sigma_n\bigl(\MEA_{n,+}\bigr).
\]
\end{proposition}

\begin{proof}
Reflection interchanges the two elements of every bilateral shell and preserves its radius. It therefore preserves the shell sequence and reverses the orientation inside each block.
\end{proof}

\section{Explicit algorithm and complexity}

The shell order can be generated with two pointers. The following pseudocode emits shell radii; each emitted radius is then expanded according to $\omega$.

\begin{quote}
\ttfamily
emit 0\\
lo $\leftarrow$ 1\\
hi $\leftarrow$ q\\
while lo $\leq$ hi do\\
\hspace*{1.5em}emit hi\\
\hspace*{1.5em}hi $\leftarrow$ hi $-1$\\
\hspace*{1.5em}if lo $\leq$ hi then\\
\hspace*{3em}emit lo\\
\hspace*{3em}lo $\leftarrow$ lo $+1$
\end{quote}

This emits
\[
0,q,1,q-1,2,q-2,\ldots.
\]
Expanding each shell produces all $n$ indices exactly once. The running time is $O(n)$ and the auxiliary storage is $O(1)$ apart from the output.

\begin{proposition}[Equivalence with the operational description]
The following operational procedure produces $\MEA_{n,\omega}$:
\begin{enumerate}[label=\arabic*.]
    \item Visit the median singleton or median pair.
    \item Visit the outermost unused bilateral pair.
    \item Visit the innermost unused bilateral pair.
    \item Alternate outermost and innermost unused pairs until none remain.
    \item Traverse every block according to the fixed orientation $\omega$.
\end{enumerate}
\end{proposition}

\begin{proof}
After the center, the shell radii are selected in the order
\[
q,1,q-1,2,q-2,3,\ldots,
\]
which is exactly outermost, innermost, next outermost, next innermost, and so on.
\end{proof}

\section{Examples}

\begin{example}[Odd cardinality: $n=7$]
Here $a=b=4$ and $q=3$. The shells are
\[
B_0=\{4\},\quad
B_1=\{3,5\},\quad
B_2=\{2,6\},\quad
B_3=\{1,7\}.
\]
The unique radius sequence is
\[
0,3,1,2.
\]
Therefore
\[
\MEA_{7,+}=(4,1,7,3,5,2,6)
\]
and
\[
\MEA_{7,-}=(4,7,1,5,3,6,2).
\]
\end{example}

\begin{example}[Odd cardinality: $n=9$]
The radius sequence is
\[
0,4,1,3,2.
\]
Hence
\[
\MEA_{9,+}=(5,1,9,4,6,2,8,3,7).
\]
\end{example}

\begin{example}[Even cardinality: $n=10$]
Here $a=5$, $b=6$, and $q=4$. The shells are
\[
\begin{aligned}
B_0&=\{5,6\}, & B_1&=\{4,7\},\\
B_2&=\{3,8\}, & B_3&=\{2,9\},\\
B_4&=\{1,10\}.&&
\end{aligned}
\]
The unique radius sequence is
\[
0,4,1,3,2.
\]
Therefore
\[
\MEA_{10,+}=(5,6,1,10,4,7,2,9,3,8)
\]
and
\[
\MEA_{10,-}=(6,5,10,1,7,4,9,2,8,3).
\]
\end{example}

\begin{table}[ht]
\centering
\caption{Selected positive-orientation MEA traversals.}
\label{tab:examples}
\begin{tabular}{c@{\qquad}l}
\toprule
$n$ & $\MEA_{n,+}$ \\
\midrule
1 & $1$ \\
2 & $1,2$ \\
3 & $2,1,3$ \\
4 & $2,3,1,4$ \\
5 & $3,1,5,2,4$ \\
6 & $3,4,1,6,2,5$ \\
7 & $4,1,7,3,5,2,6$ \\
8 & $4,5,1,8,3,6,2,7$ \\
9 & $5,1,9,4,6,2,8,3,7$ \\
10 & $5,6,1,10,4,7,2,9,3,8$ \\
\bottomrule
\end{tabular}
\end{table}

\section{Scope and nonclaims}

The theorem is deliberately narrow. Several distinctions are necessary for correct use.

\paragraph{Shell-level, not raw-index, optimization.}
The selected object at each step is a bilateral shell, and contrast is measured between shell radii. A greedy rule based on the last raw index gives a different traversal. For example, in the $n=10$ positive traversal, after the exterior block $(1,10)$ the last raw index is $10$; a raw-index farthest-neighbor rule would next prefer $2$, whereas MEA next selects the shell $\{4,7\}$ because its normalized radius $1$ is farthest from the previous shell radius $4$.

\paragraph{Local greedy characterization, not an arbitrary global optimum.}
Theorem~\ref{thm:main} proves uniqueness under a specified local maximization rule. It does not claim that MEA uniquely maximizes every possible cumulative path functional on permutations of $[n]$.

\paragraph{Orientation is a boundary condition.}
Maximal radial contrast determines the shell sequence but does not derive $\omega$. The operative direction of an application supplies the global chirality. Given $\omega$, Corollary~\ref{cor:oriented} gives uniqueness; without it, the two mirror traversals remain.

\paragraph{No semantic conclusion follows from the combinatorics alone.}
MEA determines an engagement order on a finite ordered domain. Any linguistic, semiotic, computational, or physical interpretation of that order requires separate domain-specific premises and tests.

\section{Extension to arbitrary finite ordered sets}

Every finite linearly ordered set $X=\{x_1<\cdots<x_n\}$ is order-isomorphic to $[n]$. Transporting $\MEA_{n,\omega}$ through the isomorphism $i\mapsto x_i$ gives an MEA traversal of $X$. Consequently, the construction depends only on finite linear order, cardinality, and the chosen global orientation; it does not depend on numerical labels.

\section{Conclusion}

MEA admits a parity-unified characterization. Fold a finite ordered set into bilateral shells, normalize the central shell to radius $0$, begin there, and repeatedly engage the unvisited shell having greatest radial contrast with the shell most recently engaged. The remaining radii stay contiguous, while the current radius stays adjacent to that interval; these two self-preserving invariants force a strict unique maximizer at every step. The resulting shell sequence is
\[
0,q,1,q-1,2,q-2,\ldots.
\]
Parity changes only the size of the central block. A fixed domain orientation then orders every block consistently and yields a unique index traversal. Before orientation is supplied, the only alternatives are the two mirror conjugates.

The central mathematical statement is therefore:

\begin{quote}
\emph{Given a finite linearly ordered domain with fixed orientation $\omega$, $\MEA_{\omega}$ is the unique bilateral-shell traversal that begins at the center, repeatedly engages the unvisited shell of greatest normalized radial contrast from the shell most recently engaged, and traverses every block consistently in the direction specified by $\omega$.}
\end{quote}

\end{document}